\documentclass[a4paper,11pt]{amsart}

\usepackage[OT2,T1]{fontenc}
\usepackage{eucal,fullpage,times,amsmath,amsthm,amssymb,mathrsfs,stmaryrd,color,enumerate,accents}
\usepackage[all]{xy}
\usepackage{url}
\usepackage[leqno]{amsmath}
\usepackage{amsthm}
\usepackage{amsfonts}
\usepackage{amssymb}
\usepackage{eucal}
\usepackage[all]{xy}
\CompileMatrices

\usepackage{mathrsfs}

\usepackage{amsmath,amssymb,mathrsfs,xy,amsthm,bm,url}
\input{xy}
\xyoption{all}

\usepackage{xspace}
 \usepackage[ansinew]{inputenc}
 \usepackage[T1]{fontenc}
 \usepackage{hyperref}

\theoremstyle{definition}
  \newtheorem{definition}[subsection]{Definition}
  
  \newtheorem{definition-proposition}[subsection]{Definition-Proposition}
  \newtheorem{example}[subsection]{Example}
  \newtheorem{remark}[subsection]{Remark}

   \newtheorem{assumption}[subsection]{Assumption}

\theoremstyle{theorem}
  \newtheorem{theorem}[subsection]{Theorem}
  \newtheorem{proposition}[subsection]{Proposition}
  \newtheorem*{proposition*}{Proposition}
  \newtheorem{lemma}[subsection]{Lemma}
  \newtheorem{corollary}[subsection]{Corollary}
  \newtheorem{conjecture}[subsection]{Conjecture}

\newcommand{\Cbb}{\mathbb{C}}

\newcommand{\Gbb}{\mathbb{G}}
\newcommand{\Nbb}{\mathbb{N}}
\newcommand{\Qbb}{\mathbb{Q}}
\newcommand{\Rbb}{\mathbb{R}}
\newcommand{\Sbb}{\mathbb{S}}
\newcommand{\Tbb}{\mathbb{T}}
\newcommand{\Zbb}{\mathbb{Z}}
\newcommand{\Zbhat}{{\hat{\mathbb{Z}}}}

\newcommand{\Hbf}{\mathbf{H}}
\newcommand{\Gbf}{\mathbf{G}}

\newcommand{\Pbf}{\mathbf{P}}

\newcommand{\Tbf}{\mathbf{T}}

\newcommand{\Vbf}{\mathbf{V}}

\newcommand{\mrm}{\mathrm{m}}

\newcommand{\Hscr}{\mathscr{H}}
\newcommand{\Oscr}{\mathscr{O}}

\newcommand{\Xcal}{\mathcal{X}}

\newcommand{\ad}{\mathrm{ad}}

\newcommand{\ra}{\rightarrow}
\newcommand{\lra}{\longrightarrow}

\newcommand{\mono}{\hookrightarrow}
	
\newcommand{\isom}{\cong}
\newcommand{\limproj}{\varprojlim}

\newcommand{\Hom}{\mathrm{Hom}}

\newcommand{\GSp}{\mathrm{GSp}}

\newcommand{\Ker}{\mathrm{Ker}}
\newcommand{\Lie}{\mathrm{Lie}}

\newcommand{\GL}{{\mathbf{GL}}}

\newcommand{\bsh}{\backslash}
\newcommand{\inv}{{-1}}
\newcommand{\ot}{\overset}

\newcommand{\id}{\mathrm{id}}

\newcommand{\wrt}{{with\ respect\ to}\xspace}
\newcommand{\ifof}{{if\ and\ only\ if}\xspace}
\newcommand{\cosg}{{compact\ open\ subgroup}\xspace}

\newcommand{\Gal}{\mathrm{Gal}}

\newcommand{\vbar}{{\bar{v}}}

\newcommand{\adele}{{\hat{\mathbb{Q}}}}

\newcommand{\Aut}{{\mathrm{Aut}}}

\newcommand{\Ahat}{\hat{A}}

\newcommand{\Shat}{\hat{S}}

\newcommand{\xbar}{{\bar{x}}}

\newcommand{\Vscr}{\mathscr{V}}

\newcommand{\End}{\mathrm{End}}

\newcommand{\Endbf}{\mathbf{End}}

\newcommand{\etabar}{{\bar{\eta}}}

\newcommand{\Fscr}{\mathscr{F}}

\newcommand{\Fbar}{{\bar{F}}}

\newcommand{\MT}{{\mathrm{MT}}}
\newcommand{\Der}{{\mathrm{Der}}}

\title{a note on the Manin-Mumford conjecture}

\author{Ke Chen}
\address{Department of Mathematics, University of Science and Technology of China, 230026 Hefei, Anhui Province, China}

\email{kechen@ustc.edu.cn}



\begin{document}

\begin{abstract}We prove a variant of the Manin-Mumford conjecture for abelian schemes over a normal base scheme of characteristic zero. The proof is reduced to the Manin-Mumford conjecture over fields of characteristic zero, through a theorem of Grothendieck on the endomorphisms of abelian schemes. The theorem implies a case of the Andr\'e-Oort conjecture for Kuga varieties, without resorting to the O-minimality approach nor the ergodic-Galois approach.
\end{abstract}
\maketitle
\tableofcontents

\section*{introduction}

In this paper we discuss a variant of the Manin-Mumford conjecture for abelian schemes and its relation to the Andr\'e-Oort conjecture for Kuga varieties.

\begin{conjecture}[Manin-Mumford]\label{manin mumford} Let $A$ be an abelian variety over $\Cbb$, with $(a_i)_{i\in I}$ a family of torsion points. Then the Zariski closure of $\{a_i\}_{i\in I}$ is a finite union of torsion subvarieties, i.e. subvarieties of the form $A'+a'$ where $A'\subset A$ is an abelian subvariety and $+a'$ stands for the translation by some torsion point $a'\in A(\Cbb)$.
\end{conjecture} Note that we may replace $(a_n)$ by a sequence of torsion subvarieties, because a torsion subvariety is the Zariski closure of the set of torsion points in it.

The conjecture was first proved by M. Raynaud using $p$-adic methods, cf. \cite{raynaud curve}, \cite{raynaud manin mumford}. When $A$ is defined over a number field, Faltings proved the more general Mordell-Lang conjecture which implies the Manin-Mumford conjecture, cf. \cite{faltings lang}, \cite{faltings barsotti}, as well as \cite{hindry}. There have been many other proofs, like the ergodic-Galois approach in \cite{ratazzi ullmo},   the model-theoretic approach of E. Hrushovski  \cite{bouscaren}, and the o-minimality approach by J. Pila and U. Zannier \cite{pila wilkie}.

In \cite{pink combination}, R. Pink has proposed  a conjecture for mixed Shimura varieties as a combination of the Andr\'e-Oort conjecture and the Mordell-Lang conjecture. It is further generalized into the Zilber-Pink conjecture. In this paper, we restrict our attention to a principal case of the conjecture of Pink which combines the Andr\'e-Oort conjecture with the Manin-Mumford conjecture:

\begin{conjecture}[Andr\'e-Oort conjecture for Kuga varieties]\label{andre oort conjecture for kuga varieties} Let $M$ be a Kuga variety, and let $(M_i)$ be a family of special subvarieties in $M$. Then the Zariski closure of $\bigcup_iM_i$ is a finite union of special subvarieties in $M$.
\end{conjecture}

Here Kuga varieties $M$ appear in the form of an abelian scheme $\pi:M\ra S$ where $S$ is a pure Shimura variety, typically corresponding to some moduli problem of abelian varieties, and $M$ is the universal family of abelian varieties over $S$. Special subvarieties in $M$ arise from diagrams of the following form $$\xymatrix{M' \ar[r]^\subset &  M_T \ar[d]^{\pi_{|T}} \ar[r]^\subset & M\ar[d]^\pi\\ & T\ar[r]^\subset & S  }$$ with $T\subset S$ a moduli subspace (corresponding to abelian varieties with finer additional symmetries), $M_{T}\ra T$ is the abelian $T$-scheme pulled-back from $M\ra S$, and $M'\subset M_T$ is a special subscheme in the sense of \ref{special subscheme} below, which is ''roughly'' an abelian subscheme translated by some torsion section. 

Of course one may replace Kuga varieties by general mixed Shimura varieties. But the technique and results in this paper mainly focus on abelian schemes and Kuga varieties.

There have been remarkable progresses towards the Andr\'e-Oort conjecture, cf. \cite{yafaev bordeaux} for the ergodic-Galois approach and cf. \cite{scanlon bourbaki} for a survey of the o-minimality approach of J. Pila. In the case of mixed Shimura varieties, \cite{chen kuga} has proved the equidistribution of certain families of special subvarieties in Kuga varieties, and Z. Gao has proved the Andr\'e-Oort conjecture for general mixed Shimura varieties whose pure part are Siegel modular varieties $A_g$, cf.\cite{gao}. The result of Gao is inconditional for $g\leq 6$, and relies on the GRH for CM fields when $g> 6$, as a generalization of previous results by J. Pila, J. Tsimerman, etc. 

In this paper we prove the following statement:
\begin{theorem}[main theorem]\label{main theorem} Let $\pi:M\ra S$ be a Kuga variety fibred over a pure Shimura variety, with $(M_n)$ a sequence of special subvarieties such that $\pi(M_n)=S$ for all $n$. Then the Zariski closure of $\bigcup_nM_n$ is a finite union special subvarieties whose images under $\pi$ are equal to $S$.
\end{theorem}

 It relies on a relative version of the Manin-Mumford conjecture for abelian schemes over a normal base scheme of characteristic zero. Although the arguments are elementary, even without the estimation of degrees, Galois orbits, etc., it does imply unconditionally a special case of the Andr\'e-Oort conjecture for general Kuga varieties, which is not fully covered in  \cite{chen kuga} and \cite{chen bounded}. We hope that it is useful as a footnote to the Andr\'e-Oort conjecture.

The paper is organized as follows. In Section 1, we recall the basic notions of abelian schemes, special subschemes, monodromy representations, etc. In Section 2 we prove a relative Manin-Mumford conjecture for   abelian schemes using a theorem of Grothendieck. In Section 3, we recall the basic notions of fibred Kuga varieties, and their special subvarieties. In Section 4, we use some results in Hodge theory to show that special subvarieties in Kuga variety that are faithfully flat over the base Shimura varieties are exactly the special subschemes when we view the Kuga variety as the total space of an abelian scheme, which finishes the proof.



\section{special subschemes in abelian schemes}

We recall some basic notions of abelian schemes, details of which can be found in \cite{mumford abelian}.

\begin{definition}[abelian schemes and endomorphisms]\label{abelian schemes and endomorphisms}

Let $S$ be a scheme.

(1) An abelian $S$-scheme is a proper smooth $S$-scheme $\pi:A\ra S$ equipped with a group law. The group law  is necessarily commutative, and  it is written additively.

(2) Let $A\ra S$ be an abelian $S$-scheme. We write $\End_S(A)$ for the ring of endomorphisms of the abelian $S$-scheme $A$, i.e. morphisms of the $S$-scheme $A$ respecting the group law. We write $\Endbf_S(A)$ of the \'etale sheaf $U\mapsto\End_U(A_U)$. Similarly, we have the ring of endomorphisms of $A\ra S$ up to isogeny, namely $\End_S^\circ(A):=\End_S(A)\otimes_\Zbb\Qbb$, and the \'etale sheaf $\Endbf_S^\circ(A)$. In practice we only need these sheaves on the finite \'etale sites.

(3) Let $A\ra S$ be an abelian $S$-scheme. An abelian $S$-subscheme is just a smooth closed $S$-subgroup of $A\ra S$.

\end{definition}

\begin{definition-proposition}[torsions and Tate modules]\label{torsions and tate modules}
Let $A\ra S$ be an abelian $S$-scheme of relative dimension $g$. We assume for simplicity that $S$ is connected, and we fix a geometric point $x$ of $S$. Write $\pi_1(S)=\pi_1(S,x)$ for the \'etale fundamental group of $S$ with base point $x$.

(1) For an integer $N\neq 0$, we have the endomorphism $[N]:A\ra A$, sending a section $a$ to the $N$-th multiple $a+\cdots+a$ ($N$-fold).  

The endomorphism $[N]:A\ra A$ is always flat, and its kernel $A[N]:=\Ker[N]$ is a finite flat group $S$-scheme.


When $N$ is invertible over $S$, $A[N]$ is finite \'etale over $S$. In this case, the group $\pi_1(S,x)$ acts on the fiber $A[N]_x$ respecting the group law, hence it defines a continuous representation $\rho[N]:\pi_1(S,x)\ra\GL_{\Zbb/N}(A[N]_x)$, which we call the monodromy representation of $\pi_1(S,x)$ on the $N$-torsion points. The kernel of $\rho[N]$ is a normal cofinite subgroup of $\pi_1(S,x)$ corresponding to a finite \'etale Galois covering $S_N$ of $S$.  $S_N$ is universal in the sense that if $T\ra S$ is a morphism of schemes such that in $A_T\ra T$ we have $A_T[N]\isom (\Zbb/N)^{2g}_T$ is a constant \'etale sheaf over $T$, then $T$ factors through $S_N\ra S$ uniquely.

(2) For $\ell$ a rational prime, we have the integral $\ell$-adic Tate module $\Tbb_\ell A=\limproj_nA[\ell^n]$, and the rational $\ell$-adic Tate module $\Tbb^\circ_\ell A=\Tbb_\ell A\otimes_{\Zbb_{\ell S}}\Qbb_{\ell S}$. When $\ell$ is invertible on $S$, the action of $\pi_1(S,x)$ gives a continuous $\ell$-adic representation $\rho_\ell:\pi_1(S,x)\ra\GL_{\Zbb_\ell}(\Tbb_\ell A_x)$, which is called the $\ell$-adic monodromy representation of $\pi_1(S,x)$ for $A\ra S$. Note that when $\ell$ is invertible on $S$,  $\Tbb_\ell A_x$ is isomorphic to $\Zbb_\ell^{2g}$ as a topological abelian group.

Similarly, when $S$ is of characteristic zero, we have the total Tate module $\Tbb A=\limproj_NA[N]$, and the adelic Tate module $\Tbb^\circ A=\Tbb A\otimes_{\Zbhat_S}\adele_S$. We also have the continuous monodromy representation $\rho:\pi_1(S,x)\ra\GL_{\Zbhat}(\Tbb A_x)$, with $\Tbb A_x\isom\Zbhat^{2g}$. 

In particular, the kernel of $\rho:\pi_1(S,x)\ra\GL_{\Zbhat}(\Tbb A_x)$ corresponds to a pro-finite \'etale covering $\Shat\ra S$, such that for any integer $N\neq 0$, $\Ahat[N]$ is a disjoint union of $N^{2g}$ sections, where $\Ahat\ra \Shat$ is the base change of $A\ra S$ along $\Shat\ra S$.


\end{definition-proposition}
In the rest of the paper, we assume that $S$ is an integral scheme of characteristic zero.

To formulate our main results, we need the following variants of torsion points and torsion subvarieties:

\begin{definition}[special sections and special subschemes]\label{special sections and special subschemes} Let $A\ra S$ be an abelian scheme.

(1) A special section is the image of some morphism of the form $S_N\mono A_N\ra A$, where $S_N\ra A_N$ is a torsion section of $A_N\ra S_N$ following the notations in \ref{torsions and tate modules}(1), and $A_N\ra A$ is the natural projection from the base change $A_N=A\times_SS_N$. Using finite \'etale descent, one verifies easily that special sections of $A\ra S$ are $S$-subschemes that are finite \'etale over $S$ such that after some finite \'etale base change it splits into a disjoint union of torsion sections: if $S'\subset A$ is a special section, then its preimage along some $A_N\ra A$ is the orbit of a torsion section under $\pi_1(S,x)$.

(2) A special subscheme is the image of some morphism of the form $B_N\mono A_N\ra A$ for some $N\in\Nbb_{>0}$, where $A_N=A\times_SS_N$ as in (1), and $B_N=A'_N+t_N$, where $A'_N\mono A_N$ is an abelian $S_N$-subscheme, and $t_N$ is an $N$-torsion section of $A_N\ra S_N$. Since the image of $A'_N$ in $A$ is an abelian $S$-subscheme $A'$, we may think of the special subscheme as the $\pi_1(S,x)$-orbit of the translation of $A'$ by some torsion section.

\end{definition}

When the monodromy representation is trivial, special sections are exactly torsion sections, and we have

\begin{lemma}[generic fiber]\label{geneic fiber} Let $S$ be an integral scheme of charcteristic zero, and let $A\ra S$ be an abelian $S$-scheme of relative dimension $g$. Write $\eta$ for the generic point of $S$ with function field $F$, and $\etabar$ the geometric point given by the separable closure $\Fbar$ of $F$. If the monodromy representation $\pi_1(S,\eta)\ra\GL_\Zbhat\Tbb A_\eta$ is trivial, then we have a bijection between torsion sections of $A\ra S$ and torsion points in $A_\eta$, sending a torsion section to its generic fiber.

\end{lemma}

\begin{proof}
Then the monodromy representation of $A_\eta$ factors as $\Gal(\Fbar/F)\ra\pi_1(S,\eta)\ra\GL_\Zbhat(\Tbb A_\eta)$, hence it is also trivial, and all the torsion points in $A_\eta(\Fbar)$ are defined over $F$. For each integer $N>0$, the triviality of the monodromy representation implies that $A[N]\isom(\Zbb/N)_S^{2g}$ is a constant finite \'etale group, with $A[N](S)\isom(\Zbb/N)^{2g}$. In particular, shrinking to the \'etale open $\{\eta\}\mono S$ gives the identity $A[N](S)\ra A_\eta[N](\eta)$, which is the desired bijection, $N$ being an arbitrary integer.
\end{proof}

We also have the following elementary fact:

\begin{lemma}
Let $S$ be an integral scheme of characteristic zero, and let $A\ra S$ be an abelian $S$-scheme of dimension $g$. Let $N_n$ be a sequence of positive integers tending to infinity as $n$ grows. Then the union $\bigcup_nA[N_n]$ is dense in $A$ for the Zariski topology.
\end{lemma}

\begin{proof}The structure map $A\ra S$ being of finite presentation, we may assume that $S$ is noetherian and geometrically integral.

If $S$ is a field, then we may further assume that it is embedded in $\Cbb$. Then   $\bigcup_nA[N_n](\Cbb)$ is dense in $A(\Cbb)$ for the analytic topology, hence $\bigcup_nA[N_n]$ is dense in $A$ for the Zariski topology.

For $S$ geometrically integral with generic point $\eta$ and function field $F$, it is clear that the abelian variety $A_\eta$ is dense in $A$ for the Zariski topology. Since $A[N_n]_\eta=A_\eta[N_n]$, we see that $\bigcup_nA[N_n]_\eta$ is dense in $A$, hence the density of $\bigcup_nA[N_n]$.
\end{proof}

\section{extension over a normal base}

In this section, we fix $S$ a normal integral scheme of characteristic zero, and we fix $A\ra S$ an abelian $S$-scheme of relative dimension $g$. Write $\eta$ for the generic point of $S$, and $\etabar$ its algebraic closure. Write $\pi_1(S,\etabar)\ra\GL_\Zbhat(\Tbb A_\etabar)$ for the monodromy representation, whose kernel corresponds to a profinite Galois cover $\Shat$ over $S$. Since $A\ra S$ is of finite presentation, we may assume that $S$ is locally noetherian.

Note that $\Shat\ra S$ is also normal, the proof of which is reduced to the finite \'etale case, using the following

\begin{lemma}\label{descent normal}
Let $A$ be a noraml integral ring, on which a finite group $G$ acts by automorphisms. Then the subring  $A^G$ fixed by $G$ is normal.
\end{lemma}

\begin{proof}
Write $F$ for the fraction field of $A$, and $E$ the fraction field of $A^G$. Then for any $a\in E$ integral over $A^G$, its integral equation with coefficients in $A^G$ is an integral equation over $A$, hence $a\in E\cap A=A^G$.\end{proof}

The reason we choose to work over an integral normal base of characteristiz zero is the following (cf. \cite{grothendieck} Theorem and Corollary 4.2):

\begin{theorem}[A. Grothendieck]\label{grothendieck}Let $S$  be a locally noetherian integral scheme over a field of characteristic zero, with $A,B$ two abelian $S$-schemes, $\ell$ a fixed rational prime.

(1) Let $u_\ell:\Tbb_\ell A\ra\Tbb_\ell B$ a homomorphism of integral $\ell$-adic Tate modules. If for some point $s\in S$ the homomorphism $u_\ell(s)$ comes from a homomorphism of abelian $k(s)$-schemes $u(s):A(s)\ra B(s)$, then $u_\ell$ comes from some homomorphism $u:A\ra B$, i.e. it lies in the image of the natral homomorphism $\Hom_S(A,B)\ra\Hom_{\Zbb_\ell}(\Tbb_\ell A,\Tbb_\ell B)$.

(2) Assume moreover that $S$ is normal, with $U$ an open subscheme of $S$, and $X$ an abelian $U$-scheme. Then $X$ extends to an abelian $S$-scheme $\Xcal\ra S$ \ifof $\Tbb_\ell X$ is unramified over $S$, in the sense that for any $n\in\Nbb$, $X[\ell^n]$ extends to an \'etale cover of $S$.

\end{theorem}

\begin{proposition}[constant subsheaf]\label{constant subsheaf} Let $A\ra S$ be an abelian $S$-scheme, with $S$ normal integral, such that the monodromy representation $\pi_1(S,\etabar)\ra\GL_\Zbhat(\Tbb A_\etabar)$ is trivial, i.e. $S=\Shat$. Then the sheaf $\Endbf_S(A)$ is a constant subsheaf of $\Endbf_{\Zbhat_S}(\Tbb A)$.
\end{proposition}

\begin{proof}
$\Endbf_S(A)$ is a subsheaf of $\Endbf_{\Zbhat_S}(\Tbb A)$, because for any \'etale morphism $U\ra S$, $\End_U(A_U)$ is naturally a subset of $\End_{\Zbhat_U}(\Tbb A_U)$: if a morphism $f:A_U\ra A_U$ sends each $N$-torsion subgroup to zero, then it sends the closure of $\bigcup_NA_U[N]$ to zero, namely it is zero as an endomorphism of $A_U$  over $U$.

For the constancy, we first show that any geometric point $x$ over $\eta$ gives an isomorphism $\tau:\End_S(A)\ra\End_x(A_x)$ by restriction. The injectivity is proved as above, and for the surjectivity,  we have the commutative diagram $$\xymatrix{\End_S(A)\ar[r]\ar[d] & \End_x(A_x)\ar[d]\\ \End_{\Zbhat_S}(\Tbb A)\ar[r]& \End_\Zbhat(\Tbb A_x)}$$ where the horizontal map on the bottom is bijective due to the triviality of the monodromy representation. The two vertical maps are inclusions, hence the horizontal map $\tau$ on the top row is surjective, using \ref{grothendieck} (1).




Now for any \'etale morphism $U\ra S$ with a geometric point $x$ in $U$, the monodromy representation $\pi_1(U,x)\ra \GL_\Zbhat(\Tbb (A_U)_x)$ is trivial, hence $\End_U(A_U)\ra\End_x(A_x)$ is bijective.  Hence $\End_S(A)\ra\End_U(A_U)$ is an isomorphism for all $U$, which proves the constancy.
\end{proof}

For an abelian variety we can realize its abelian subvariety as the neutral component of the kernel of some endomorphism, based on the following:

\begin{theorem}[splitting theorem, cf. \cite{conrad} 3.19, 3.20]\label{splitting} 
Let $k$ be a  field, and let $X$ be an abelian variety over $k$. Then for any abeian subvariety $Y\subset X$, there exists an abelian subvariety $Z\subset X$ such that the product map $Y\times Z\ra X$ is an isogeny.
\end{theorem}

In fact let $Y\subset X$ be an abelian subvariety, with $N$ the degree of the isogeny $Y\times Z\ra X$ given by the theorem.  The multiplication $[N]:Y\times Z\ra Y\times Z$ factors through some isogeny $(p_Y,p_Z):X\ra Y\times Z$, and the composition $$X\ot{(p_Y,p_Z)}\lra Y\times Z\ot{i_Y,i_Z}\lra X\times X\ot{m_X}\ra X$$ is an isogeny, with $i_Y$ and $i_Z$ inclusions of abelian subvarieties. In particular, the composition $\psi:=m_X\circ(0\times i_Z)\circ p_Z\in\End_k(X)$ is an endomorphism, whose kernel contains $Y$ as the neutral component.

\begin{corollary} Let $A\ra S$ etc. be as in the beginning of this section, with $S$ normal integral. If the monodromy representation is trivial, then every abelian subvariety $A'$ in the generic fiber $A_\eta$ extends to an abelian $S$-subscheme $B$ of $A\ra S$ with $B_\eta=A'$.
\end{corollary}

\begin{proof}

Let $A'\subset A_\eta$ be an abelian subvariety. Then by \ref{splitting} we can find some endomorphism $\phi:A_\eta\ra A_\eta$ such that $A'$ is equal to the neutral component of the closed subgroup variety $\Ker\phi$. The constancy of $\Endbf_S(A)$ shows that $\phi$ extends to a unique endomorphism $\Phi$ of $A\ra S$. The kernel $\Ker\Phi$ is a closed $S$-subgroup of $A\ra S$, and it is smooth over $S$ because it is the pull-back of the neutral section $S\mono A$ along $\Phi$. Therefore the neutral component of $\Ker\Phi$, denoted as $B$, is a closed $S$-subscheme of $A$ and is an abelian $S$-subscheme under the group law of $A\ra S$. Taking generic fiber we see that $B_\eta$ is a connected subgroup variety of $\Ker\Phi_\eta=\Ker\phi$, namely it is equal to $A'$.
\end{proof}

We also have the following 

\begin{corollary}[descent to finite level]\label{descent to finite level}
Let $A\ra S$ be an abelian $S$-scheme, with $S$ normal integral of characteristic zero. Let $\Shat\ra S$ be the  pro-finite \'etale Galois covering corresponding to the kernel of the monodromy representation, and let $A'\ra \Shat$ be an abelian  $\Shat$-subscheme of $\Ahat:=A\times_S\Shat$. Then $A'$ descends to some finite \'etale cover $T\ra S$, i.e. there exists a finite \'etale cover $T\ra S$ such that $\Shat\ra S$ factors as $\Shat\ra T\ra S$ and that $A'=B_{\Shat}$ where $B$ is an abelian $T$-subscheme of the base change $A_T\ra T$.
\end{corollary}

\begin{proof}
This is the standard reduction of projective limits: $A'\subset\Ahat$ is the neutral component of $\Ker\phi$ for some endomorphism $\phi:\Ahat\ra\Ahat$. Since the projective limit $\Shat=\limproj S_N$ is taken over the filtrant system $(S_N)$ with $S_N$ corresponding to the kernel of $\pi_1(S,\etabar)\ra\GL_{\Zbb/N}(A[N]_\etabar)$, there exists some integer $N>0$ such that $\phi:\Ahat\ra\Ahat$ is pulled-back from some endomorphism $\Phi:A_N\ra A_N$ with $A_N=A\times_SS_N$, and that $\Ker\Phi$ has neutral component $B$ such that $B$ is an abelian $S_N$-subscheme with $B_{\Shat}=A'$. One may thus take $T=S_N$.
\end{proof}

We proceed to prove the Manin-Mumford conjecture in the relative setting using special subschemes.

\begin{proposition}[relative Manin-Mumford]\label{relative Manin-Mumford}Let $A\ra S$ be an abelian $S$-scheme, with $S$ a normal integral scheme of characteristic zero. Let $A_n$ be a sequence of special subschemes of $A\ra S$. Then the Zariski closure of $\bigcup_nA_n$ can be represented as a finite union of special subschemes.
\end{proposition}

\begin{proof}
Since $A\ra S$ is of finite presentation, we may assume for simplicity that $S$ is geometrically integral of generic point $\eta$, with $\etabar$ the geometric point in $S$ corresponding to the spearable closure of $\eta$.

(1) We first consider the case when the monodromy representation is trivial, i.e. $\Shat=S$. In this case, we have proved that taking base change from $S$ to $\eta$ gives \begin{itemize}

\item a bijection between torsion sections of $A\ra S$ and torsion points of $A_\eta$;

\item a bijection between abelian $S$-subschemes of $A\ra S$ and abelian subvarieties of $A_\eta$.

\end{itemize} 
And taking Zariski closure gives inverses to these bijections because $\eta$ is dense in $S$.

Special subschemes in $A$ are of the form $a+B$ with $a$ a torsion section and $B$ an abelian $S$-subscheme. Let $A_n$ be a sequence of special subschemes of $A\ra S$. Then $A_{n,\eta}$ is a torsion subvariety of $A_\eta$, with $A_{n,\eta}$ dense in $A_n$ for the Zariski closure topology. The closure of $\bigcup_nA_{n,\eta}$ in $A_\eta$ is a finite union of torsion subvarieties, whose closure is a finite union of special subscheme in $A$. The Manin-Mumford conjecture is thus immediate in this case.

(2) In general, a special subscheme of $A$ is the image $\pi'(A')$ where 

\begin{itemize}
\item  $\pi':A_N\ra A$ is the projection for some base change $A_N\ra S_N$, $S_N$ corresponding to the kernel of $\pi_1(S,\etabar)\ra\GL_{\Zbb/N}(A[N]_\etabar)$.;
\item $A'=a'+B'$ for some torsion section $a'$ of $A_N\ra S_N$ and some abelian $S_N$-subscheme $B'$ of $A_N$.
\end{itemize}
Hence a special subscheme is of the form $\pi(A')$, where $\pi:\Ahat\ra A$ is the projection for the base change $\Ahat\ra \Shat$, with $\Shat$ corresponding to the kernel of $\pi_1(S,\etabar)\ra\GL_{\Zbhat}(\Tbb A_\etabar)$, and $A'=a'+B'$ for some torsion section $a'$ of $\Ahat\ra \Shat$ and $B'$ some abelian $\Shat$-subscheme. 

The projection $\pi:\Ahat\ra A$ is a pro-finite cover, and in particular it is universally closed. Let $A_n$ be a special subscheme in $A$ of the form $\pi(B_n)$ with $B_n$ a special subscheme of $\Ahat\ra \Shat$. Then the Zariski closure of $\bigcup_nA_n$ contains $\pi(B)$, with $B$ the Zariski closure of $\bigcup_nB_n$ in $\Ahat$. By (1) we know that $B$ is a finite union of special subschemes in $\Ahat$, hence $\pi(B)$ is a finite union of special subschemes in $A$, hence it is equal to the Zariski closure of $\bigcup_nA_n$.\end{proof}

\section{preliminaries on Kuga varieties}

 We recall briefly the definitions of Kuga data, Kuga varieties, and their special subvarieties, cf. \cite{chen kuga} Section 2.
 
 \begin{definition}[Kuga data]\label{kuga data}
 A Kuga datum is a pair $(\Pbf,Y)$ given by some $(\Gbf,X;\Vbf)$ as follows
 
 \begin{itemize}
 \item $(\Gbf,X)$ is a pure Shimura datum in the sense of \cite{deligne pspm}; in particular, $X$ is a $\Gbf(\Rbb)$-conjugacy class of homomorphisms $x:\Sbb\ra\Gbf_\Rbb$ subject to some algebraic constraints;
 
 \item $\rho:\Gbf\ra\GL_\Vbf$ is an  algebraic representation on a finite-dimensional $\Qbb$-vector space such that for any $x\in X$ the composition $\rho\circ x:\Sbb\ra \GL_{\Vbf,\Rbb}$ is a Hodge structure of type $\{(-1,0),(0,-1)\}$.
 
 \end{itemize}
 
 We put $\Pbf=\Vbf\rtimes\Gbf$ and $Y=\Vbf(\Rbb)\times X$, with $Y$ viewed as a $\Pbf(\Rbb)$-conjugacy class of homomorphisms $y:\Sbb\ra\Pbf_\Rbb$ subject to some algebraic constraints. In the language of \cite{chen bounded}, $(\Pbf,Y)=\Vbf\rtimes(\Gbf,X)$ is fibred over $(\Gbf,X)$.
 
 When $\Vbf=0$, we get (pure) Shimura data.
 
 For simplicity, we also require that the Kuga data are irreducible in the sense of \cite{pink thesis} 2.13, which means that for any $\Qbb$-subgroup $\Hbf\subsetneq\Gbf$ there is some $x\in X$ such that $x(\Sbb)\nsubseteq\Hbf_\Rbb$.

 \end{definition}
 
 \begin{definition}[morphisms and subdata]\label{morphisms and subdata}
 
  A morphism between Kuga data is of the form $(f,f_*):(\Pbf,Y)\ra(\Pbf',Y')$ with $f:\Pbf\ra\Pbf'$ a homomorphism of $\Qbb$-groups, and $f_*:Y\ra Y'$ is the push-forward sending $y:\Sbb\ra\Pbf_\Rbb$ to $f\circ y:\Sbb\ra\Pbf'_\Rbb$.

 A subdatum of $(\Pbf,Y)$ is a morphism of Kuga data $(f,f_*):(\Pbf_1,Y_1)\ra(\Pbf,Y)$ such that both $f$ and $f_*$ are inclusions of subsets.
 
 Let $(\Pbf,Y)=\Vbf\rtimes(\Gbf,X)$ be a Kuga datum. The natural map $(\Pbf,Y)\ra(\Gbf,X)$ is a morphism of Kuga data, which we call the natural projection of $(\Pbf,Y)$ onto its pure base: $\Gbf$ is the maximal reductive quotient of $\Pbf$. The Levi decomposition $\Pbf=\Vbf\rtimes\Gbf$ also extends to an inclusion of subdatum $(\Gbf,X)\mono(\Pbf,Y)$ which we call the pure section corresponding to $\Pbf=\Vbf\rtimes\Gbf$.
 
 Note that for a Kuga datum $(\Pbf,Y)=\Vbf\rtimes(\Gbf,X)$, $Y$ is a complex manifold with a transitive action of $\Pbf(\Rbb)$, and the natural projection $Y\ra X$ is a holomorphic vector bundle, equivariant \wrt $\Pbf(\Rbb)\ra\Gbf(\Rbb)$. The fiber $\pi^\inv x$ is the real vector space $\Vbf(\Rbb)$ with the complex structure defined by $x:\Sbb\ra\Gbf_\Rbb\ra\GL_{\Vbf,\Rbb}$.

 \end{definition}
 
 \begin{definition}[connected Kuga varieties]\label{connected Kuga varieties} We write $\adele=\Zbhat\otimes_\Zbb\Qbb$ for the ring of finite adeles.
 
 (1) The (complex) Kuga variety defined by the Kuga datum $(\Pbf,Y)$ at level $K$ for some \cosg $K\subset\Pbf(\adele)$ is a double quotient of the form $$M_K(\Pbf,Y)(\Cbb)=\Pbf(\Qbb)\bsh (Y\times\Pbf(\adele)/K)$$ with $\Pbf(\Qbb)$ acts on $Y\times\Pbf(\adele)/K$ through the diagonal. Take $\Pbf(\Qbb)_+$ the stablizer in $\Pbf(\Qbb)$ of some connected component $Y^+\subset Y$, we have $$M_K(\Pbf,Y)(\Cbb)=\coprod_a\Gamma_K(a)\bsh Y^+$$ with $\Gamma_K(a)=\Pbf(\Qbb)_+\cap aKa^\inv$, $a$ running through a set of representatives of the finite double quotient $\Pbf(\Qbb)_+\bsh\Pbf(\adele)/K$. 
 
 The general theory of mixed Shimura varieties in \cite{pink thesis} shows that the set $M_K(\Pbf,Y)(\Cbb)$ defined above are quasi-projective normal varieties over $\Cbb$, and they admits canonical models over certain number fields.  In this paper we only treat them as normal algebraic varieties over $\Cbb$.
 
 The map $\wp_\Pbf:Y\times\Pbf(\adele)/K\ra M_K(\Pbf,Y)(\Cbb), \ (y,aK)\mapsto[y,aK]$ is called the (complex) uniformization.
 
 (2) A connected Kuga datum is of the form $(\Pbf,Y;Y^+)$ with $(\Pbf,Y)=\Vbf\rtimes(\Gbf,X)$ a Kuga datum and $Y^+\subset Y$ a connected component of $Y$. Note that $Y^+$ is homogeneous under $\Pbf(\Rbb)^+$. We also have $(\Pbf,Y;Y^+)=\Vbf\rtimes(\Gbf,X;X^+)$ in the sense of \ref{kuga data}, with $X^+$ the image of $Y^+$ in $X$ which is a connected component of $X$.
 
 Connected Kuga varieties are quasi-projective algebraic varieties over $\Cbb$ of the form $\Gamma\bsh Y^+$ with $\Gamma\subset\Pbf(\Qbb)_+$ some congruence subgroup. They also admit canonical models over some number fields.
 
 We write $\wp_\Gamma$ for  the uniformization map $Y^+\mapsto \Gamma\bsh Y^+, y\mapsto \Gamma y$.
 
 (3) In particular, when we write $(\Pbf,Y;Y^+)=\Vbf\rtimes(\Gbf,X;X^+)$ and take a congruence subgroup of the form $\Gamma=\Gamma_\Vbf\rtimes\Gamma_\Gbf$, with $\Gamma_\Vbf\subset\Vbf(\Qbb)$ and $\Gamma_\Gbf\subset\Gbf(\Qbb)_+$ congruence subgroups such that $\Gamma_\Vbf$ is stabilized by $\Gamma_\Gbf$, then we have  the natural projection  $\pi:M=\Gamma\bsh Y^+\ra S=\Gamma_\Gbf\bsh X^+$, which is an abelian $S$-scheme with neutral section $S\mono M$ given by $(\Gbf,X;X^+)\mono(\Pbf,Y;Y^+)$.
  \end{definition}
  
  
  
  
  \begin{assumption}\label{assumption}
 Unless otherwise mentioned, we will always assume that $\Gamma_\Gbf$ is a torsion-free congruence subgroup of $\Gbf(\Qbb)_+$. In this case $S$ is smooth, and the natural map $\Gamma_\Gbf'\bsh X^+\ra\Gamma_\Gbf\bsh X^+$ is finite \'etale for any congruence subgroup $\Gamma_\Gbf'\subset\Gamma_\Gbf$. Sicne $S$ is also normal by \cite{baily borel}, we see that the \'etale fundamental group of $S$ is equal to the pro-finite completion of $\Gamma$, the image of $\Gamma_\Gbf$ inside $\Aut(X^+)\isom\Gbf^\ad(\Rbb)^+$, which only differs from $\Gamma_\Gbf$ by a central subgroup.
 
\end{assumption}

\begin{remark}[group law]\label{group law} Let $(\Pbf,Y)=\Vbf\rtimes(\Gbf,X)$ be a fibred Kuga datum.  We write the group law on $\Pbf=\Vbf\rtimes\Gbf$ as $$(v,g)\cdot(v',g')=(v+g(v'),gg')$$ for local sections $v,v'\in\Vbf$, $g,g'\in\Gbf$, with $g(v')=gv'g^\inv=\rho(g)(v')$ by the representation $\rho:\Gbf\ra\GL_\Vbf$. In particular, for $u\in\Vbf$, we have $u(v,g)u^\inv=(u,1)(v,g)(-u,1)=(v+u-g(u),g)$. 

Write $\pi:(\Pbf,Y)\ra(\Gbf,X)$ for the natural projection, then the fibred product $(\Pbf,Y)\times_{(\Gbf,X)}(\Pbf,Y)$ exists as a fibred Kuga datum, which is simply $(\Vbf\oplus\Vbf)\rtimes(\Gbf,X)$. The sum $\Vbf\oplus\Vbf\ra\Vbf$ defines a group law $(\Pbf,Y)\times_{(\Gbf,X)}(\Pbf,Y)\ra(\Pbf,Y)$ with $(\Gbf,X)\ra(\Pbf,Y)$ as the neutral section. On $\Pbf=\Vbf\rtimes\Gbf$ it writes as $(v,g)+(v',g)=(v+v',g)$ and on $Y$ it writes as $(v,x)+(v',x)=(v+v',x)$. Fix a connected component $X^+\subset X$, its pre-image $Y^+=\pi^\inv X^+\subset Y$, and congruence subgroups $\Gamma_\Gbf\subset\Gbf(\Qbb)_+$, $\Gamma_\Vbf\subset\Vbf(\Qbb)$ (stabilized by $\Gamma_\Gbf$) and $\Gamma=\Gamma_\Vbf\rtimes\Gamma_\Gbf$, we see that $M=\Gamma\bsh Y^+\ra S=\Gamma_\Gbf\bsh X^+$ is a bundle of compact Lie group over $S$: $(\vbar,\xbar)+(\vbar',\xbar)=(\overline{v+v'},\xbar)$ for $(v,x),(v',x)\in\pi^\inv x$, $x\in X^+$. 

The fibers are compact complex tori, and $M\ra S$ is an abelian $S$-scheme as the variation of Hodge structures given by the monodromy representation $\pi_1(S)\ra\GL_{\Gamma_\Vbf}$ is polarized, due to the universal property of $(\Gbf,X)$ mentioned later in \ref{moduli of Hodge structures}; see also \cite{deligne pspm}, \cite{pink thesis} and \cite{pink combination}.

\end{remark}
 
 \begin{definition}[special subvarieties and Hecke translates]\label{special subvarieties and Hecke translates} For $M=\Gamma\bsh Y^+$ a connected Kuga variety defined by $(\Pbf,Y;Y^+)$ as above, a special subvariety in $M$ is of the form $\wp_\Gamma(Y'^+)$ given by some subdatum $(\Pbf',Y';Y'^+)\subset(\Pbf,Y;Y^+)$. Note that we require $Y'^+$ to be a connected component of $Y'$ contained in $Y^+$.
 
 Take $q\in\Pbf(\Qbb)_+$, $q\Gamma q^\inv$ remains a congruence subgroup of $\Pbf(\Qbb)_+$, and we have an isomorphism $\tau_q:M=\Gamma\bsh Y^+\ra q\Gamma q^\inv\bsh Y^+$, $\Gamma\cdot y\mapsto q\Gamma q^\inv \cdot qy$, called the Hecke tranaslation by $q$. Note that when $q\in\Vbf(\Qbb)$, $(\Pbf,Y;Y^+)=\Vbf\rtimes(q\Gbf q^\inv, qX;qX^+)$, and $\tau_q$ sends the pure section of $M\ra S$ to the pure section of $M'=q\Gamma q^\inv\bsh Y^+\ra S'=q\Gamma_\Gbf q^\inv\bsh qX^+$ given by $(q\Gbf q^\inv,qX;qX^+)\mono(\Pbf,Y;Y^+)$.
 
 Of course we can also talk about more general Hecke translation given by $q\in\Pbf(\adele)$, cf. \cite{chen kuga}. 
 \end{definition}

 The following proposition describes subdata and special subvarieties in an explicit way as we have seen in Introduction.
 \begin{proposition}[description of subdata and special subvarieties]\label{description of subdata and special subvarieties} 
 (1) Let $(\Pbf,Y)=\Vbf\rtimes(\Gbf,X)$ be a Kuga datum fibred over a pure Shimura datum $(\Gbf,X)$. Then a Kuga subdatum $(\Pbf', Y')\subset(\Pbf,Y)$ is of the form $(\Pbf',Y')=\Vbf'\rtimes(v\Gbf'v^\inv, vX')$ where $(\Gbf',X')$ is a pure Shimura subdatum of $(\Gbf,X)$, $\Vbf'$ is a subrepresentation of $\Gbf'$ in $\Vbf$, and $v\in\Vbf(\Qbb)$ conjugate $\Gbf'$ into a Levi $\Qbb$-subgroup $v\Gbf' v^\inv$ of $\Pbf'$. For a fixed $(\Pbf',Y')$, $v$ is unique up to translation by $\Vbf(\Qbb)$.
 
 (2) Let $M=\Gamma\bsh Y^+$ be a connected Kuga variety defined by $(\Pbf,Y;Y^+)=\Vbf\rtimes(\Gbf,X;X^+)$ with $\Gamma=\Gamma_\Vbf\rtimes\Gamma_\Gbf$. The natural projection $\pi:M\ra S=\Gamma_\Gbf\bsh X^+$ defines an abelian $S$-scheme, and the special subvariety $M'$ defined by $(\Pbf',Y';Y'^+)=\Vbf'\rtimes(v\Gbf'v^\inv, vX';vX'^+)$ fits into the diagram $$\xymatrix{M'\ar[r]^\subset & M_{S'}\ar[r]^\subset\ar[d]^\pi& M\ar[d]^\pi\\ & S'\ar[r]^\subset & S}$$ where $S'=\wp_{\Gamma_\Gbf}(X'^+)$ is the pure special subvariety in $S$ defined by $(\Gbf',X';X'^+)$, $M_{S'}$ is the pull-back of $M\ra S$ along $S'\ra S$, equal to the special subvariety defined by $\Vbf\rtimes(\Gbf',X';X'^)$. $M'$ is a torsion subscheme of the abelian $S'$-scheme $M_{S'}\ra S'$ in the sense of \ref{special subscheme}: the subdatum $\Vbf'\rtimes(\Gbf',X';X'^+)$ defines an abelian $S'$-subscheme $A'_{S'}$, and $(v\Gbf'v^\inv, vX';vX'^+)$ defines a special section of $M_{S'}\ra S'$, the ''translation'' by which gives the torsion subscheme $M'$.
 
 \end{proposition}
 
 \begin{proof}
 
 The part (1) is from \cite{chen kuga} 2.6 and 2.10. We only outline how (2) is interpreted via the special subschemes. We may thus assume that $S'=S$.
 
 Write $\Gamma_\Gbf(v)=\{g\in\Gamma_\Gbf:v-g(v)\in\Gamma_\Vbf\}$. Then $v\Gamma_\Gbf(v)v=\Gamma_\Vbf\rtimes\Gamma_\Gbf\cap v\Gamma_\Gbf v^\inv$. Base change to $f:T=\Gamma_\Gbf(v)\bsh X^+\ra S$, we get the abelian $T$-scheme $M_T=(\Gamma_\Vbf\rtimes\Gamma_\Gbf(v))\bsh Y^+\ra T$. Aside from the neutral section $T\mono M_\Tbf$ given by $(\Gbf,X;X^+)\subset(\Pbf,Y;Y^+)$ and $\Gamma_\Gbf(v)\subset\Gamma_\Vbf\rtimes\Gamma_\Gbf(v)$, we also have the pure special subvariety $T(v)=\wp_\Gamma(vX^+)$ corresponding to $(v\Gbf v^\inv,vX;vX^+)$. Since we have shrinked to $\Gamma_\Gbf(v)$, the equality $v(\Gamma_\Vbf\rtimes\Gamma_\Gbf(v))v^\inv=\Gamma_\Vbf\rtimes v\Gamma_\Gbf(v)v^\inv$ implies that $T(v)=v\Gamma_\Gbf(v)v^\inv\bsh vX^+\isom\Gamma_\Gbf(v)\bsh X^+$, and thus $T(v)$ is a torsion section, whose torsion order is the minimal integer $N>0$ such that $N\cdot v\in\Gamma_\Vbf$. The subdatum $\Vbf'\rtimes(\Gbf,X;X^+)$ defines an abelian $T$-subscheme of $M_T$, whose translation by $T(v)$ is a torsion subscheme of $M_T$. Its image under $M_T\ra M$ is a special subscheme of $M$, which is exactly the special subvariety $M'=\wp_\Gamma(Y'^+)$.
 \end{proof}

\section{special subschemes in Kuga varieties}

In this section we show that special subschemes in a fibred Kuga variety $M\ra S$ are special subvarieties that are faithfully flat over $S$. The proof makes use of some facts from the theory of variation of Hodge structures, details of which can be found in \cite{deligne pspm}, \cite{milne introduction}, \cite{moonen lecture} etc. We adopt standard abbreviations such as ''HS'' for Hodge structures, ''PVHS'' for polarized variation of Hodge structures, etc.

\begin{theorem}[abelian schemes vs. variation of Hodge structures, \cite{deligne hodge II} 4.4.3(a)]\label{abelian schemes vs. variation of Hodge structures} Let $S$ be a smooth scheme over $\Cbb$ of finite type. Then we have the equivalence between the following two categories:\begin{itemize}
\item (1) the category of abelian $S$-schemes (with morphisms respecting the group laws);
\item (2) the category of polarizable variation of integral Hodge structures ($\Zbb$-PVHS) of type $\{(-1.0),(0,-1)\}$.
\end{itemize}

\end{theorem} 
 The equivalence sends an abelian $S$-scheme $f:A\ra S$ to the $\Zbb$-PVHS $\Hscr=\Hscr(A/S)$ whose underlying local system of $\Zbb$-modules is dual to $R^1f_*\Zbb_A$, with $\Hscr_s=H_1(A_s,\Zbb)$ as the fiber at $s$.  The exponential map realize $A$ as the the quotient sheaf $$0\ra\Hscr\ra\Lie_SA\ra A\ra 0$$ where $\Lie_SA$ is the sheaf of ''vertical tangents'' of $A\ra S$, i.e. the pull-back of the relative tangent sheaf $\Der_SA$ along the neutral section $S\mono A$. The Hodge decomposition in the relative setting is $$0\ra F^0\ra\Hscr\otimes_{\Zbb_S}\Oscr_S\ra\Lie_SA\ra 0$$ with $F^0$ the 0-th piece of the Hodge filtration.
 
 Note that when we fix $A\ra S$ an abelian $S$-scheme, the equivalence above also implies the bijection between
 
 \begin{itemize}
 \item (1)' abelian $S$-subschemes of $A$;
 
 \item (2)'  sub-variation of rational Hodge structures of $\Hscr_\Qbb=\Hscr\otimes_{\Zbb_S}\Qbb_S$
 \end{itemize}
 which sends an abelian $S$-subscheme $A'$ to $\Hscr(A'/S)_\Qbb$.  Conversely, given $\Hscr'_\Qbb$ and object in (2)', $\Hscr':=\Hscr'_\Qbb\cap\Hscr$ is a $\Zbb$-PVHS of type $\{(-1,0),(0,-1)\}$ which defines an abelian $S$-scheme $A'$, and the eveident map $\Hscr'\mono\Hscr$ shows that $A'$ is an abelian $S$-subscheme of $A$.
 
 \bigskip
 

Deligne showed in \cite{deligne pspm} that a pure Shimura datum $(\Gbf,X)$ is universal in the following sense:
 
 \begin{theorem}[moduli of Hodge structures]\label{moduli of Hodge structures} Let $(\Gbf,X)$ be a pure Shimura datum. Then the composition $w:\Gbb_\mrm\mono\Sbb\ot{x}\ra\Gbf_\Rbb$ is a central cocharacter, independent of $x\in X$. For any (algebraic) representation $\rho:\Gbf\ra\GL_\Vbf$ over $\Qbb$ such that $\rho\circ w:\Gbb_\mrm\ra\GL_\Vbf$ is some central cocharacter $t\mapsto t^k\id_\Vbf$ defined over $\Qbb$, the constant local system $\Vscr$ on $X$ with fiber $\Vbf(\Qbb)$ underlies a unique $\Qbb$-PVHS
 
 \end{theorem}
 
We also need the notion of (generic) Mumford-Tate groups, cf. \cite{moonen lecture}. 
 
\begin{definition}[Mumford-Tate groups]\label{Mumford-Tate groups}

(1) For $(V,h:\Sbb\ra\GL_{V,\Rbb})$ a $\Qbb$-HS, its Mumford-Tate group is the smallest $\Qbb$-subgroup $\Gbf$ of $\GL_V$ such that $h(\Sbb)\subset\Gbf_\Rbb$. $\Gbf$ is connected. If the Hodge structure is polarizable, then $\Gbf$ is reductive. We write $\Gbf=\MT(h)$.

If $W$ is a space of tensors on $V$, i.e. a subquotient of  $\oplus_iV^{\otimes m_i}\otimes (V^\vee)^{\otimes n_i}$ ($m_i,n_i\in\Nbb$) , then $W$ is a $\Qbb$-HS for the natural action of $\Sbb$ \ifof it is stabilized by the natural action of $\Gbf$ by the tensor constructions. In particular, writing $W^{0,0}$ for the subspace of $W_\Cbb$ fixed by $\Sbb_\Cbb$, then $W\cap W^{0,0}$ equals $W^\Gbf$ the $\Qbb$-subspace fixed by $\Gbf$, and this space is called the space of Hodge class of type $(0,0)$ in $W$.


(2) Let $S$ be a complex manifold, and $(\Vscr,\Fscr)$ a $\Qbb$-VHS on $S$. We fix a model $V$ for $\Vscr$, i.e. a $\Qbb$-vector space such that for each $x\in S$ we have an isomorphism $V\isom \Vscr_x$, and that for any $x,y\in S$, the induced isomorphism $\Vscr_x\isom V\isom\Vscr_y$ is induced by a prescribed path in $S$ from $x$ to $y$. Typically we fix a base point $s\in S$ and a path $\ell_x$ from $s$ to $x$ for each $x$, so that $V=\Vscr_s$ and $V\isom\Vscr_x$ is given by $\ell_x$.

For each $x\in S$, we have the Mumford-Tate group at $x$, i.e. the Mumford-Tate group $\MT(\Vscr_x)$, which is identified as a $\Qbb$-subgroup of $\GL_V$ via the isomorphism $V\isom\Vscr_x$. There exists a countable union of analytic subspaces $\Sigma=\bigcup_nS_n$ and a $\Qbb$-subgroup $\Gbf\subset\GL_V$ such that $\Gbf=\MT(\Vscr_x)$ for any $x\notin \Sigma$. For $x\in\Sigma$, we have $\MT(\Vscr_x)\subsetneq \Gbf$. $\Gbf$  is called the generic Mumford-Tate group of the $\Qbb$-VHS $(\Vscr,\Fscr)$. When the $\Qbb$-VHS is polarizable, $\Gbf$ is reductive.

\end{definition}

\begin{remark}
In general, the Mumford-Tate group of a $\Qbb$-HS $(V,h)$ is a $\Qbb$-subgroup of $\GL_V\times\Gbb_\mrm$ so that the Hodge classes of $(p,p)$-type ($p\in\Zbb$) can be studied in the same way as in the above definition. In this paper we will only need Hodge classes of type $(0,0)$ and the above definition suffices. 
\end{remark}

\begin{example}[Kuga-Siegel case]\label{Kuga-Siegel case}

Let $(\Pbf,Y)=\Vbf\rtimes(\Gbf,X)$ be a Kuga datum. Then the representation $\Gbf\ra\GL_\Vbf$ defines a $\Qbb$-VHS $\Vscr$ on $X$ whose underlying local system is the constant sheaf of fiber $\Vbf(\Qbb)$. By \ref{moduli of Hodge structures}, this $\Qbb$-VHS is polarizable. Since the local system is constant, the polarization is given by some symplectic form $\psi:\Vbf\otimes\Vbf\ra\Qbb(-1)$ which $\Gbf$ preserves up to similitude. Hence the Kuga datum $(\Pbf,Y)$ is equivalently given by a homomorphism of pure Shimura data $(\Gbf,X)\ra(\GSp_\Vbf,\Hscr_\Vbf)$.

The image of $(\Gbf,X)\ra(\GSp_\Vbf,\Hscr_\Vbf)$ is a subdatum  $(\Gbf',X')\subset(\GSp_\Vbf,\Hscr_\Vbf)$, which is also irreducible as $(\Gbf,X)$ already is. It follows immediately from the definition \ref{Mumford-Tate groups} that $\Gbf'$ is the generic Mumford-Tate group of $\Vscr$ on $X$.

Take a lattice $\Gamma_\Vbf$ in $\Vbf(\Qbb)$, and a torsion-free congruence subgroup $\Gamma_\Gbf\subset\Gbf(\Qbb)_+$ stabilizing $\Gamma_\Vbf$, we get the connected Shimura variety $S=\Gamma_\Gbf\bsh X^+$. The representation $\Gamma_\Gbf\ra\GL(\Gamma_\Vbf)$ defines a $\Zbb$-PVHS, as $\Gamma_\Gbf$ acts on $X^+$ through the fundamental group of $S$, and the $\Qbb$-PVHS associated to it is obviously $\Vscr$. The abelian $S$-scheme corresponding to this $\Zbb$-PVHS is exactly the fibred Kuga variety $M=\Gamma\bsh Y^+\ra S$ with $\Gamma=\Gamma_\Vbf\rtimes\Gamma_\Gbf$ and $Y^+=\Vbf(\Rbb)\times X^+$.

\end{example}

In the rest of this section, we fix $\pi:M\ra S$ an abelian scheme given by a fibred connected Kuga variety $M=\Gamma\bsh Y^+$, defined by the datum $(\Pbf,Y)=\Vbf\rtimes(\Gbf,X)$, with a torsion-free congruence subgroup $\Gamma=\Gamma_\Vbf\rtimes\Gamma_\Gbf$, and $S=\Gamma_\Gbf\bsh X^+$. If $M'\subset M$ is a special subvariety such that $\pi(M')=S$, then by \ref{description of subdata and special subvarieties} we see that $M$ is a special subscheme of $M\ra S$. We proceed to prove the inverse:

\begin{theorem}[special subschemes vs. special subvarieties]\label{special subschemes vs. special subvarieties} Let $M\ra S$ be defined by $(\Pbf,Y)=\Vbf\rtimes(\Gbf,X)$ and $\Gamma=\Gamma_\Vbf\rtimes\Gamma_\Gbf$ as above. Let $M'\subset M$ be a special subscheme. Then $M'$ is a special subvariety, and $\pi(M')=S$.

\end{theorem}

\begin{proof}
The pure Shimura variety $S=\Gamma\bsh X^+$ is normal. For any non-zero integer $N\in\Nbb$, write $\Gamma_\Gbf(N)=\Ker(\Gamma_\Gbf\ra\GL(\Gamma_\Vbf)\ra\GL(\Gamma_\Vbf/N\Gamma_\Vbf)$, then $\Gamma_\Gbf(N)$ is a congruence subgroup in $\Gamma_\Gbf$, and the base change $$\pi_N:M_N=\Gamma_\Vbf\rtimes\Gamma_\Gbf(N)\bsh Y^+\ra S_N=\Gamma_\Gbf(N)\bsh X^+$$ is an abelian $S_N$-scheme in which the $N$-torsion subgroup split, i.e. $M_N[N]=\coprod_{v}M_N(v)$, where\begin{itemize}

\item the disjoint union is indexed by $\frac{1}{N}\Gamma_\Vbf/\Gamma_\Vbf$, which is the $N$-torsion subgroup of $\Gamma_\Vbf\bsh\Vbf(\Rbb)$;

\item for $v\in\Vbf(\Qbb)$, $M_N(v)$ stands for the special subvariety defined by $(v\Gbf v^\inv, vX;vX^+)$. 
\end{itemize}

Note that for general $v\in\Vbf(\Qbb)$, the special subvariety $M_N(v)$ only depends on the class of $v$ in $\Gamma_\Vbf\bsh\Vbf(\Qbb)$, and the resulting special subvariety is $\Gamma'\bsh vX^+$, with $\Gamma'=\Gamma_\Vbf\rtimes\Gamma_\Gbf\cap v\Gbf(\Qbb)_+v^\inv$, which is equal to $v\Gamma_\Gbf(v)v^\inv$ with $$\Gamma_\Gbf(v)=\{g\in\Gamma_\Gbf:g(v)-v\in\Gamma_\Vbf\}.$$ Since $\Gamma_\Gbf(N)$ is the kernel of $\Gamma_\Gbf\ra\GL(\Gamma_\Vbf/N)\isom\GL(\frac{1}{N}\Gamma_\Vbf/\Gamma_\Vbf)$, we see that $g\in\Gamma_\Gbf$ always fixes the class of $v$ modulo $\Gamma_\Vbf$ when $v\in\frac{1}{N}\Gamma_\Vbf$, hence $M_N(v)\isom\Gamma_\Gbf(N)\bsh X^+=S_N$. This isomorphism is actually the Hecke translation given by $v$, using $v(\Gamma_\Vbf\rtimes\Gamma_\Gbf(N))v^\inv\isom\Gamma_\Vbf\rtimes\Gamma_\Gbf(N)$, cf. \ref{special subvarieties and Hecke translates} and \ref{group law}.

By the definition of special subschemes, it remains to show that every abelian $S$-subscheme of $M\ra S$ is a special subvariety $M'$ such that $\pi(M')=S$. It suffices to treat the problem for a sufficiently small level $\Gamma_\Gbf$, so by shrinking $\Gamma_\Gbf$ we may assume that the sheaf of endomorphism algebra $\Endbf_S(M)$ is constant, as we have seen that over the integral normal scheme $S$ the sheaf is locally constant and its generic fiber is a finite rank $\Zbb$-algebra.   Since we only need to study the neutral component of the kernel of endomorphisms, we may replace $\Endbf_S(M)$ by the isogeny algebra $\Endbf^\circ_S(M)$, which is also constant.

Passing to isogeny from the equivalence in \ref{abelian schemes vs. variation of Hodge structures}, the sheaf $\Endbf^\circ_S(M)$ is the same as the endomorphism sheaf of  the $\Qbb$-PVHS $\Hscr_\Qbb=\Hscr(M/S)\otimes_{\Zbb_S}\Qbb_S$, namely the sheaf associated to the Hodge classes of type $(0,0)$ in $\Endbf(\Hscr_\Qbb)\isom\Hscr_\Qbb\otimes_{\Qbb_S}\Hscr_\Qbb^\vee$ with $\Hscr_\Qbb^\vee$ the $\Qbb$-PVHS dual to $\Hscr_\Qbb$. Since $\Hscr_\Qbb$ is given by the representation $\Gbf\ra\GL_\Vbf$, $\Endbf(\Hscr_\Qbb)$ is given by the tensor representation $\Gbf\ra\GL_{\End(\Vbf)}$, and the $(0,0)$ part corresponds to the trivial subrepresentation $\End(\Vbf)^\Gbf$. So the constant sheaf $\Endbf^\circ_S(M)$ is associated to the vector space $\Endbf(\Vbf)^\Gbf$. 

Let $M'$ be an abelian $S$-subscheme, realized as the neutral component of some $\phi\in\Endbf_S^\circ(M)$. We thus identify $\phi$ as an element of $\End(\Vbf)^\Gbf$, and it follows from the equivalences \ref{abelian schemes vs. variation of Hodge structures} and the characterization of abelian subschemes via sub-$\Qbb$-PVHS that $M'$ corresponds to the $\Qbb$-PVHS $\Hscr'$ given by $\Vbf'$ which is the kernel of $\phi:\Vbf\ra\Vbf$. Clearly $\Vbf'$ is a subrepresentation of $\Gbf$ in $\Vbf$, and for any $x\in X$, the action of $\Sbb$ on $\Vbf'$ through $x$ makes $\Vbf$ a sub-$\Qbb$-HS of $\Vbf$, hence is of type $\{(-1,0),(0,-1)\}$. We obtain a Kuga subdatum $(\Pbf',Y')=\Vbf'\rtimes(\Gbf,X)$, which defines an abelian $S$-subscheme, whose associated $\Qbb$-PVHS is the one given by the action of $\Gbf$ on $\Vbf'$. Therefore this abelian $S$-subscheme is equal to $M'$, and $M'$ is a special subvariety, faithfully flat over $S$ under $\pi$.
\end{proof}

We immediately get the desired variant of the Manin-Mumford conjecture in the Kuga setting
\begin{corollary}Let $\pi:M\ra S$ be a Kuga variety fibred over a pure Shimura variety $S$ as an abelian $S$-scheme. Let $(M_n)$ be a sequence of special subvarieties faithfully flat over $S$, i.e. $\pi(M_n)=S$ for all $n$. Then the Zariski closure of $\bigcup_nM_n$ is a finite union of special subvarieties faithfully flat over $S$.
\end{corollary}

\end{document}